\documentclass[12pt,reqno]{amsart}
\usepackage{amsmath,amssymb}
\usepackage{amsthm}
\usepackage{thmtools}
\usepackage{mathtools}
\usepackage{dsfont}
\usepackage{color}
\usepackage{enumitem}
\usepackage{hyperref}
\usepackage[inner=1.0in,outer=1.0in,bottom=1.0in, top=1.0in]{geometry}

\usepackage{cases}

\newcommand{\Z}{{\mathbb{Z}}}

\newcommand{\R}{{\mathbb{R}}}
\newcommand{\C}{{\mathbb{C}}}
\def\spt{{\mathrm{spt}}}
\DeclareMathOperator{\re}{Re}
\DeclareMathOperator{\im}{Im}

\newcommand\be{\begin{equation}}
\newcommand\ee{\end{equation}}
\newcommand\bee{\begin{equation*}}
\newcommand\eee{\end{equation*}}
\newcommand\ben{\begin{enumerate}}
\newcommand\een{\end{enumerate}}
\newcommand{\defeq}{\vcentcolon=}
\def\SL{{\rm SL}}

\renewcommand{\(}{\left(}
\renewcommand{\)}{\right)}
\newcommand{\la}{\left|}
\newcommand{\ra}{\right|}

\newtheorem{theorem}{Theorem}

\newtheorem{lemma}[theorem]{Lemma}
\newtheorem{corollary}[theorem]{Corollary}
\newtheorem{proposition}[theorem]{Proposition}

\theoremstyle{remark}

\numberwithin{equation}{section}
\numberwithin{theorem}{section}
\numberwithin{lemma}{section}
\numberwithin{proposition}{section} 
\numberwithin{example}{section}
\numberwithin{definition}{section}
\numberwithin{corollary}{section}

\author{Oscar E. Gonz\'alez}
\address{Department of Mathematics, University of Illinois at Urbana-Champaign, Urbana, IL 61801}
\email{oscareg2@illinois.edu}
\title{
Effective estimates for the smallest parts function}
\date{\today}

\begin{document}
\begin{abstract}
We give a substantial improvement for the error term in the asymptotic formula for the smallest parts function  $\spt(n)$ of Andrews.
Our methods depend on an explicit bound for sums of Kloosterman sums of half integral weight on the full modular group.
 \end{abstract}
\maketitle

\section{Introduction}   
The smallest parts function $\spt(n)$, introduced by Andrews \cite{andrews}, is defined for any integer $n\geq1$ as the number of 
smallest parts among the integer partitions of $n$.  For example, the partitions of $n=4$ are (with the smallest parts underlined)
\begin{align*}
&\underline{4},\\
&3+\underline{1},\\
&\underline{2}+\underline{2},\\
&2+\underline{1}+\underline{1},\\
&\underline{1}+\underline{1}+\underline{1}+\underline{1},
\end{align*}
and so $\spt(4)=10$.  
Apart from its combinatorial significance, this function is also of interest because
 the generating function is closely related to
a weak harmonic Maass form (see \eqref{eq:harmonicF}),
and it has been the topic of much recent study.
Define
\be
\label{eq:lambda}
\lambda(n) \defeq \frac{\pi}{6} \sqrt{24n-1}.
\ee
Refining an asymptotic result of Bringmann \cite{bringmann}, Locus Dawsey and Masri 
used the algebraic formula for the 
smallest parts function (\cite[Thm. 2]{AA})
and traces of singular moduli
to 
prove the following asymptotic formula for the smallest parts function $\spt(n)$.
\begin{theorem}[\cite{LM}, Thm. 1.1]
\label{SptLM} 
 Let $\lambda(n)$ be as in \eqref{eq:lambda}. 
 Then for all $n \geq 1$, we have 
\bee
\spt(n) = \frac{\sqrt{3}}{\pi \sqrt{24n-1}} e^{\lambda(n)} + E_{s}(n),
\eee
where 
\bee
|E_s(n)| < (3.59 \times 10^{22}) 2^{q(n)}(24n-1)^2 e^{\frac{\lambda(n)}{2}}
\eee
and
\bee
q(n) \defeq \frac{\log(24n-1)}{|\log(\log(24n-1))-1.1714|}.
\eee
\end{theorem}
In Theorem \ref{thm:sptbound} we give a substantial improvement to 
Theorem \ref{SptLM}. Our methods rely on
 the exact formula \eqref{eq:AA} and an explicit bound for sums of Kloosterman sums (Theorem \ref{thm:KloostermanBound}).
\begin{theorem}
\label{thm:sptbound}
 Let $\lambda(n)$ be as in  $\eqref{eq:lambda}$.
 Then for all $n \geq 1$, we have 
\bee
\spt(n) = \frac{\sqrt{3}}{\pi \sqrt{24n-1}} e^{\lambda(n)} + E_{s}(n),
\eee
where 
\bee
|E_s(n)| < 4.1 e^{\frac{\lambda(n)}{2}}.
\eee
\end{theorem}
In Figure \ref{fig:table} we present some data 
that shows how close $E_s(n)$ is to the bound given in
 Theorem \ref{thm:sptbound}.
 The values of $\spt(n)$ were obtained 
by using the recurrence given in \cite[Thm.  1]{AArecurrence}.
\begin{figure}[h!]
\label{fig:table}
\[\def\arraystretch{1.2}
\begin{array}{c||c||c||c}
n& \spt(n) & |E_s(n)|& 4.1 e^{\frac{\lambda(n)}{2}} \\ 
\hline
\hline
1\, 000& 6.0\times 10^{32} & 2.1\times 10^{15} & 1.7\times 10^{18} \\
10\, 000& 2.8\times 10^{108} & 8.0\times 10^{52} & 2.1\times 10^{56} \\
100\, 000& 6.8\times 10^{348} & 7.0\times 10^{172} & 5.7\times 10^{176} \\
1\, 000\, 000& 1.1\times 10^{1110} & 1.6\times 10^{553} & 4.1\times 10^{557} \\
5\, 000\, 000& 5.0\times 10^{2486} & 2.2\times 10^{1241} & 1.3\times 10^{1246} \\
\end{array}
\]
\caption{
}
\end{figure}

Using our method we can obtain more terms in the asymptotic expansion of $\spt(n)$.
For example, we can prove the following theorem.
\begin{theorem}
\label{thm:sptbound2}
 Let $\lambda(n)$ be as in  $\eqref{eq:lambda}$.
 Then for all $n \geq 1$, we have 
\bee
\spt(n) = 
\frac{\sqrt{3}}{\pi  \sqrt{24 n-1}} e^{\lambda(n)}
+\frac{(-1)^n \sqrt{6}  }{\pi  \sqrt{24 n-1}} e^{\frac{\lambda(n)}{2}} +E_{s_2}(n)
\eee
where 
\bee
|E_{s_2}(n)| < 8 e^{\frac{\lambda(n)}{3}}.
\eee
\end{theorem}

In 2014, Chan and Mao \cite{ChanMao} raised the following question 
(later stated as a conjecture in \cite{Chen})
regarding $\spt(n)$:
$$ \frac{\sqrt{6}}{\pi}\sqrt{n}\,p(n)<\spt(n)<\sqrt{n}\,p(n).$$ 
Locus Dawsey and Masri \cite[Thm.  1.3]{LM} proved the following stronger result:
for each $\varepsilon>0$, there is an $N(\varepsilon) > 0$ 
such that for all $n\geq N(\varepsilon)$ we have
$$ \frac{\sqrt{6}}{\pi}\sqrt{n}\,p(n)<\spt(n)<\(\frac{\sqrt{6}}{\pi} + \varepsilon\)\sqrt{n}\,p(n).$$\
In the following corollary we improve this result.
\begin{corollary}
\label{cor:spt}
Let $\lambda(n)$ be as in  $\eqref{eq:lambda}$.
 Then for all $n \geq 1$, we have 
$$\spt(n) = \frac{\sqrt{24n-1}}{2\pi} p(n) + \frac{6\sqrt{3}}{\pi^2 (24n-1)} e^{\lambda(n)} + E(n)$$
where
\bee
|E(n)| < 4.11 e^{\frac{\lambda(n)}{2}}.
\eee
\end{corollary}

Much of the interest in the smallest parts function arises from its connection to a harmonic Maass form. 
Let $\eta$ be the Dedekind eta function 
\be
\label{eq:eta-def}
	\eta(z) \defeq e\(\frac{z}{24}\)\prod_{n=1}^\infty \(1-e(nz)\), \qquad \im(z)>0
\ee
with $e(x) \defeq e^{2\pi i x}$.
Define a weak harmonic Maass form of weight 
$3/2$ on $\operatorname{SL}_{2}(\mathbb{Z})$
with multiplier $\overline{\chi}$  by
\begin{align}
\label{eq:harmonicF}
\nonumber
F(z) &\defeq \sum_{n=1}^{\infty} \spt(n) q^{n-\frac{1}{24}} 
-\frac{1}{12} \cdot \frac{E_{2}(z)}{\eta(z)} 
+ \frac{\sqrt{3i}}{2\pi} \int_{-\overline{z}}^{i\infty} \frac{\eta(w)}{(\tau+w)^{\frac{3}{2}}}dw\\
&=\sum_{n=0}^{\infty} S(n)q^{n-\frac{1}{24}} + \frac{\sqrt{3i}}{2\pi} 
\int_{-\overline{z}}^{i\infty} \frac{\eta(w)}{(z+w)^{\frac{3}{2}}}dw.
\end{align}
Here $\chi$ is as in \eqref{eq:eta-mult-def},
and $E_2$ is the usual weight two Eisenstein series given by 
   \bee
   \label{eq:E2def}
   E_2(z) \defeq 1-24\sum_{n=1}^{\infty} \sigma_{1}(n) q^n,
   \eee
where $\sigma_{1}(n) \defeq \sum_{d\mid n} d$.
We prove the following effective asymptotic formula for $S(n)$.
\begin{theorem}
\label{thm:S}
For all $n\geq 1$ we have
\bee
S(n) = 2\sqrt{3}e^{\lambda(n)}+E_{S}(n)
\eee
where $\lambda(n)$ is as in \eqref{eq:lambda}
 and
\bee
E_{S}(n) \leq 44.11 e^{\frac{\lambda(n)}{2}}.
\eee
\end{theorem}
This improves \cite[Thm. 1.4]{LM}, where a bound of size
\bee
(4.30 \times 10^{23}) 2^{q(n)} (24n - 1)^2    e^{\frac{\lambda(n)}{2}}
\eee
was obtained (with $q(n)$ as in Theorem \ref{SptLM}).

Our methods rely on an explicit bound for the sums  $\sum_{c\leq x} \frac{A_c(n)}{c}$, where the
Kloosterman sum $A_c(n)$ is given by 
  \be
  \label{eq:AcDef}
  A_c(n)\defeq \sum_{\substack{d \bmod c\\ (d,c)=1}}e^{\pi i s(d,c)}e^{-2\pi i\frac{dn}c},
  \ee
and $s(d,c)$ is the Dedekind sum defined by 
\be 
  \label{eq:ded-sum-def}
	s(d,c) \defeq \sum_{r=1}^{c-1} \frac{r}{c} \(\frac{dr}{c} - \left\lfloor \frac{dr}{c} \right\rfloor - \frac{1}{2}\).	
\ee
These sums exhibit good cancelation. 
We give a brief summary of known bounds for sums of such Kloosterman sums.
For individual Kloosterman sums Lehmer  \cite[Thm. 8]{lehmer} proved
\begin{equation}
\label{eq:lehmerAc}
|A_c(n)| < 2^{\omega_o(c)} c^{\frac{1}{2}} \leq  \tau(c) c^{\frac{1}{2}},
\end{equation} 
where $\omega_o(c)$ is the number of distinct odd primes dividing $c$
and 
$\tau(c)$ is the number of divisors of $c$.
Using \eqref{eq:lehmerAc}
one obtains \bee
\sum_{c\leq x} \frac{A_c(n)}{c}\ll_{\epsilon} x^{\frac{1}{2}+\epsilon}.
\eee
 The work of Goldfeld-Sarnak \cite{gs} yields 
 \be
 \label{eq:gsBound}
 \sum_{c\leq x} \frac{A_c(n)}{c} \ll_{n,\epsilon} x^{\frac{1}{6}+\epsilon}.
 \ee
Ahlgren-Andersen \cite{AAimrn} 
(with improvements by Dunn in the $n$-aspect \cite{dunn}) replaces the bound in
 \eqref{eq:gsBound} by $\ll_{\epsilon}  \(x^\frac16+n^\frac14\)(nx)^\epsilon$.
It is conjectured (generalization of Linnik-Selberg) that 
the bound can be replaced by
$\ll_{\epsilon} (nx)^\epsilon$. 
Our methods depend on an explicit version of the work of Goldfeld-Sarnak and Pribitkin for sums of Kloosterman sums of half integral weight on the full modular group.
 
  For any $\delta > 0$ we have
 $2^{\omega_o(c)} \ll_{\delta} c^{\delta}$. For $\delta>0$ let 
$\ell(\delta)$  be a constant such that for all $c\in\mathbb{N}$ we  have
 \be
 \label{eq:ell}
2^{\omega_o(c)} \leq \ell(\delta) c^{\delta}.
\ee 
Then we have the following bound.
   \begin{theorem}
\label{thm:KloostermanBound}
Let $0<\delta \leq 1/4$. For any $x\geq 1$ and any
integer $n\geq1$  we have
\bee
 \la \sum_{c\leq x} \frac{A_c(n)}{c} \ra
\leq  \(652.33 \zeta^{2}(1+\delta) \tau((24n-23)^2)  |\log\delta| (n-1/24)^{\frac{1}{4}}+3 \ell(\delta) \log x \) 
x^{\frac{1}{6}+\delta},
\eee
where $\ell(\delta)$ is as in \eqref{eq:ell}.
   \end{theorem}
   For example, we may take $\ell(1/4) = 8.447$ and $\ell(1/5) = 28.117$. 
This follows from $2^{\omega_o(c)} \leq \tau(c)$
  and \cite[page 221]{highly}. 
  For $s>1$ we have
\begin{equation}
  \label{prop:ZetaSq}
    \zeta^2(s) = \sum_{n=1}^{\infty} \frac{\tau(n)}{n^s}.
  \end{equation}
The special case of Theorem \ref{thm:KloostermanBound} with $\delta=1/4$ is given
  as the following corollary.
\begin{corollary}[$\delta=1/4$]
\label{ex:KloostermanBound1/4}
\bee
 \la \sum_{c\leq x} \frac{A_c(n)}{c}\ra
\leq 
 \(19094.8 \; \tau((24n-23)^2)    (n-1/24)^{\frac{1}{4}}+25.35 \log x \)
 x^{\frac{5}{12}}.
\eee
   \end{corollary}

In the next section we give some background material. In Section \ref{sec:inner} we calculate the inner product of two Poincar\'e series in order to obtain an expression for the Kloosterman zeta function. We write the inner product as a main term plus an error term $Q_{1,n}$. In 
Section~\ref{sec:boundsForQ} we obtain a bound for the error term. In Section \ref{sec:boundsForUm} we give bounds for the norm of the Poincar\'e series. In order to do this we require some results on the $K$ Bessel functions. In Section \ref{sec:boundsForKZeta} we prove a bound on the Kloosterman zeta function.
Theorem \ref{th:KZbound}
 is a quantitative version of Pribitkin's main theorem \cite{pribitkin}. In Section 
\ref{sec:proofOfMainTh}  we prove
Theorem \ref{thm:KloostermanBound}
 using the bound on the Kloosterman zeta function and the Phragm\'en-Lindel\"of principle.
Finally in Section \ref{sec:proofOfSptbound} we use Theorem \ref{thm:KloostermanBound} and the exact formula \eqref{eq:AA}
 for the smallest parts function to prove Theorem \ref{thm:sptbound}
 and Corollary \ref{cor:spt}.

\section{Preliminaries}
Let $\Gamma = \Gamma_0(N)$ for some $N\geq 1$.
We say that $\nu:\Gamma\to \C^\times$ is a multiplier system of weight $k\in \mathbb{R}$ if
\begin{enumerate}[label=(\roman*)]
	\item $|\nu|=1$,
	\item $\nu(-I)=e^{-\pi i k}$, and
	\item $\nu(\gamma_1 \gamma_2) \, j(\gamma_1\gamma_2,\tau)^k=\nu(\gamma_1)\nu(\gamma_2) \, j(\gamma_2,\tau)^k j(\gamma_1,\gamma_2\tau)^k$ for all $\gamma_1,\gamma_2\in \Gamma$.
\end{enumerate}
If $\nu$ is a multiplier system of weight $k$, then it is also a multiplier system of weight $k'$ for any $k'\equiv k\pmod 2$, and the conjugate $\overline\nu$ is a multiplier system of weight $-k$.
Define $\alpha_{\nu} \in [0,1)$ by the condition
$\nu\( \begin{psmallmatrix} 
	1 & 1 \\ 
	0 & 1
	\end{psmallmatrix} \) = e(-\alpha_{\nu})$.
For $n\in \Z$, define \ $n_{\nu} \defeq n-\alpha_{\nu}$.
The Kloosterman sum for a general multiplier $\nu$ is given by 
\be
\label{eq:kloos_def}
	S(m,n,c,\nu) := \sum_{\substack{0\leq a,d<c \\ \gamma=\begin{psmallmatrix} a&b\\ c&d 
	\end{psmallmatrix}
	\in \Gamma}} \overline\nu(\gamma) e\(\frac{m_\nu a+n_\nu d}{c}\).
\ee

We are interested in the multiplier system $\chi$ of weight $1/2$ on $\SL_2(\Z)$ given by
\begin{equation} \label{eq:eta-mult-def}
	\eta(\gamma z) = \chi(\gamma)\sqrt{cz+d}\,\eta(z), \qquad \gamma=
	\begin{pmatrix} 
	a & b \\ 
	c & d
	\end{pmatrix} \in \SL_2(\Z),
\end{equation}
where $\eta$ is as in
\eqref{eq:eta-def}. 
Rademacher (see (74.11), (74.12), and (71.21) of \cite{rademacherbook}) showed that 
for $\gamma=\begin{psmallmatrix}
 a&b\\ c&d
 \end{psmallmatrix}$ with $c>0$ we have 
\begin{equation} \label{eq:chi-dedekind-sum}
	\chi(\gamma) = \sqrt{-i} \, e^{-\pi i s(d,c)} \, e\(\frac{a+d}{24c}\),
\end{equation}
where $s(d,c)$ is as in \eqref{eq:ded-sum-def}.
From \eqref{eq:chi-dedekind-sum}  we have 
$\chi\(  \begin{psmallmatrix} 
	1 & 1 \\ 
	0 & 1
	\end{psmallmatrix}\) = e(1/24)$,
	so
$$
	\alpha_{\chi} = \tfrac{23}{24} \quad \text{and} \quad \alpha_{\bar\chi} = \tfrac{1}{24}.
$$
For the eta-multiplier, \eqref{eq:kloos_def}
and \eqref{eq:chi-dedekind-sum} give
\begin{equation*}
	S(m,n,c,\chi) = \sqrt i \, \sum_{\substack{d\bmod c \\ (d,c)=1}} e^{\pi i s(d,c)} 
	e\(\frac{(m-1)\overline d+(n-1)d}{c}\),
\end{equation*}
so the sums $A_c(n)$ are given by
\begin{equation} \label{eq:A-c-n-S}
	A_c(n) = \sqrt{-i} \, S(1,1-n,c,\chi).
\end{equation}
Recently, Ahlgren and Andersen gave 
the following Rademacher-type exact formula for the 
smallest parts function as a conditionally convergent infinite sum of $I$-Bessel functions and Kloosterman sums (\cite[Thm. 1]{AA}):
\be
\label{eq:AA}
\spt(n)=\frac{\pi}{6}(24n-1)^\frac{1}{4}\sum_{c=1}^\infty\frac{A_c(n)}{c}\left(I_{\frac{1}{2}}-I_{\frac{3}{2}}\right)\left(\frac{\pi\sqrt{24n-1}}{6c}\right).
\ee
We also have (\cite{AA})
\be
\label{eq:AASn}
S(n) = 2 \pi (24n-1)^{\frac{1}{4}} \sum_{c=1}^{\infty} 
\frac{A_c(n)}{c} I_{\frac{1}{2}} 
\( \frac{\pi \sqrt{24n-1}}{6c} \),
\ee
and 
\be
\label{eq:sptSp}
\spt(n)= \frac{1}{12} S(n) - \frac{24n-1}{12} p(n).
\ee

\section{The inner product $I_{m,n}(s,w)$}
\label{sec:inner}
We obtain an expression for the inner product of two Poincar\'e series by unfolding.
Let $z=x+iy \in \mathbb{H}$ and $s=\sigma+it \in \mathbb{C}$. For $m>0$, 
define the Poincar\'e series $\mathcal{U}_{m}(z, s, \frac{1}{2}, \chi)$ by
\be
\mathcal{U}_{m}(z, s, \frac{1}{2}, \chi)  \defeq \sum_{\gamma \in \Gamma_{\infty}\backslash \Gamma}
\overline{\chi(\gamma)} j(\gamma,z)^{-\frac{1}{2}} \im(\gamma z)^s e(m_{\chi}\gamma z), \qquad \sigma>1,
\ee
where $e(x) = e^{2\pi i x}$. 
Selberg \cite{selberg} proved that $\mathcal{U}_{m}(z, s, \frac{1}{2}, \chi)$ has an analytic continuation to a meromorphic
function.
Let $n\leq 0$ and define
\be
Z_{m,n}(s) \defeq \sum_{c>0}\frac{S(m,n,c,\chi)}{c^{2s}}.
\ee 
Note that $Z_{m,n}(s)$ converges absolutely for $\re(s)>1$ and is analytic in this half-plane. Selberg proved 
that $Z_{m,n}(s)$ has an analytic continuation to a function which is meromorphic in the whole plane.

Define the Petersson inner product by
\be
\label{eq:inner_prod}
	\langle f,g \rangle := \int_{\Gamma\backslash\mathbb{H}} f(\tau) 
	\overline{g(\tau)} \, \frac{dx\,dy}{y^2}.
\ee
Let $\re(w)>1$. Define
\be
I_{m,n}(s,w) \defeq 
\Big\langle   
\mathcal{U}_{m}(z, s, \tfrac{1}{2}, \chi), \overline{\mathcal{U}_{1-n}(z, w, -\tfrac{1}{2}, \overline{\chi}) }
\Big\rangle
\ee
and
  \begin{multline}
  \label{eq:Qdef}
Q_{m,n}(s,w) \defeq Z_{m,n}(s) \int_{0}^\infty  y^{w-s-1} e^{2\pi n_{\chi} y}\\
\times\(\int_{-\infty}^\infty    \frac{(u+i)^{-\frac{1}{2}}}{(u^2+1)^{s- \frac{1}{4}} }
 e(-n_{\chi}yu) 
\(e\(\frac{-m_{\chi}}{c^2 y(u+i)}\)-1\) \, du\) \, dy.
\end{multline}
We have the following expansion for the inner product for
$m_\chi>0$ and $n_{\chi}<0$. The case $m_\chi>0$ and $n_{\chi}>0$ is given in \cite{pribitkin}.
\begin{lemma}
\label{lem:innerprod}
Let $m_\chi>0$ and $n_{\chi}<0$. For 
$s=\sigma+it$ with $\sigma>1$ and $\re(w)>\sigma$ we have
\bee
I_{m,n}(s,w)
=
 \frac{\Gamma(w+s-1)  \Gamma(w-s)}{\Gamma\(w+\frac{1}{4}\) \Gamma\(s-\frac{1}{4}\)}
(-1)^{-s-w}  Z_{m,n}(s)  \pi^{s-w+1}
 (-i)^{-\frac{1}{2}}    ( n_{\chi})^{s-w}
  4^{-w+1} +Q_{m,n}(s,w).
 \eee
\end{lemma}

\begin{proof}
From the proof of \cite[Lemma 4.2]{AAimrn}, we have
\bee
I_{m,n}(s,w) = Z_{m,n}(s)\int_{-\infty}^\infty 
\(\frac{u+i}{|u+i|}\)^{-\frac{1}{2}}(u^2+1)^{-s}
\int_0^\infty y^{w-s-1}e\(\frac{- m_\chi}{c^2 y(u+i)}-  n_\chi y u\)e^{2\pi  n_\chi y}\, dydu.
\eee
Therefore,
\begin{align*}
I_{m,n}(s,w)&=
Z_{m,n}(s)\int_{0}^\infty  y^{w-s-1} e^{2\pi n_{\chi} y} \int_{-\infty}^\infty    
\frac{(u+i)^{-\frac{1}{2}}}{(u^2+1)^{s-\frac{1}{4}} }
e\(\frac{-m_{\chi}}{c^2 y(u+i)}\) e(-n_{\chi}yu) \, du dy\\
 &=
Z_{m,n}(s)\int_{0}^\infty  y^{w-s-1} e^{2\pi n_{\chi} y}
\int_{-\infty}^\infty    \frac{(u+i)^{-\frac{1}{2}}}{(u^2+1)^{s-\frac{1}{4}} }
 e(-n_{\chi}yu) 
  \, du dy + Q_{m,n}(s,w).
\end{align*}
Next we use the formula (\cite[3.384 \#9]{table})
\begin{multline*}
\int_{-\infty}^{\infty} (b+ix)^{-2\mu} (c-ix)^{-2\nu} e^{-ipx} dx\\
=2\pi(b+c)^{-\mu-\nu} \frac{(-p)^{\mu+\nu-1}}{\Gamma(2\mu)} \exp\(\frac{b-c}{2} p\) 
W_{\mu-\nu,\frac{1}{2}-\nu-\mu}(-bp-cp),
\end{multline*}
with parameters $p=2\pi n_{\chi} y$, $b=c=1$, $\mu=\frac{s}{2}-\frac{1}{8}$, and $\nu = \frac{s}{2}+\frac{1}{8}$.
 Here $W$ is the Whittaker function.
We obtain
\bee
\int_{-\infty}^\infty    \frac{(u+i)^{-\frac{1}{2}}}{(u^2+1)^{s-\frac{1}{4}} }
 e(-n_{\chi}yu) 
  \, du
 = 
\frac{(-1)^{2s} \pi (-i)^{-\frac{1}{2}}  (-\pi n_{\chi}y)^{s-1}}{\Gamma\(s-\frac{1}{4}\)} 
 W_{-\frac{1}{4},s-\frac{1}{2}}(-4 \pi n_{\chi}y).
\eee
It follows that
 \begin{align*}
  & I_{m,n}(s,w)-Q_{m,n}(s,w)\\
   &=   
  Z_{m,n}(s)(-1)^{-2s} 
  \frac{\pi (-i)^{-\frac{1}{2}}  }{\Gamma\(s-\frac{1}{4}\)} (-\pi n_{\chi})^{s-1}
  \int_{0}^\infty  y^{w-2} e^{2\pi n_{\chi} y}
W_{-\frac{1}{4},s-\frac{1}{2}}(-4 \pi n_{\chi}y)  \, dy\\
 &=   
  Z_{m,n}(s) (-1)^{-2s} 
  \frac{\pi (-i)^{-\frac{1}{2}}  }{\Gamma\(s-\frac{1}{4}\)} (-\pi n_{\chi})^{s-1}
  (-4\pi n_{\chi})^{-w+1}  \int_{0}^\infty  y^{w-2} e^{-\frac{y}{2}}
W_{-\frac{1}{4},s-\frac{1}{2}}(y)  \, dy.
  \end{align*}
 Now we use \cite[\S 1.13, 13.52]{MellinTable}
 to obtain 
\bee
 \int_{0}^\infty 
      y^{w-2} e^{-\frac{y}{2}}
 W_{-\frac{1}{4}, \,s-\frac{1}{2}}(y)  \, dy=
\frac{\Gamma(w+s-1)\Gamma(w-s)}{\Gamma\(w+\frac{1}{4}\)}.
\eee
Lemma \ref{lem:innerprod} follows.
\end{proof}

 \section{Bounds for Q}
 \label{sec:boundsForQ} 
 In this section we obtain a bound for the error term $Q_{1,n}(s,s+2)$ in Lemma \ref{lem:innerprod}.
We begin with a preliminary lemma. 
\begin{lemma}
\label{lem:eineq}
Let $\re(z)\leq 0$. Then,
$$|e^z - 1 | \leq 1.682 |z|^{\frac{1}{4}}.$$
\end{lemma}
\begin{proof}
Let $f(z) = \frac{e^{z}-1}{z^{\frac{1}{4}}}$, $f(0)=0$.
On $\re(z) = 0$ we have $$|f(z)| = \frac{2|\sin(y/2)|}{|y^{\frac{1}{4}}|} 
\leq \max\(\frac{2}{|y|^{\frac{1}{4}}}, \frac{|y|}{|y|^{\frac{1}{4}}}\).$$
Thus, $|f(z)| \leq 2^{\frac{3}{4}}$ on $\re(z) = 0$. 
Since the same bound holds trivially on $\re(z) = n$ with $n\leq -2$, 
the result follows by the Phragm\'en-Lindel\"of principle.
\end{proof}

\begin{lemma}
\label{lem:Qbound}
Let $s = \sigma+it$ with $\sigma =1/2+\delta/2$ and $0<\delta\leq 1/4$, and let $n\leq 0$ be an integer. Then \bee
|Q_{1,n}(s,s+2)| \leq 0.414  \zeta^{2}(1+\delta)
|n_{\chi}|^{-\frac{7}{4}}.
\eee
\end{lemma}

\begin{proof}
From \eqref{eq:Qdef} we have 
\begin{align*}
|Q_{1,n}(s,s+2)|   &\leq  |Z_{1,n}(s)|\int_{0}^\infty  
 y e^{2\pi n_{\chi} y}
 \int_{-\infty}^\infty    \frac{1}{(u^2+1)^{\sigma} }
\; \;   \Bigg|   e\(\frac{-1}{24c^2 y(u+i)}\) -1 
\Bigg|\, du dy.
\end{align*}
By Lemma \ref{lem:eineq} we obtain 
\begin{align*}
 |Q_{1,n}(s,s+2)|&\leq 
1.682 |Z_{1,n}(s)|\int_{0}^\infty  y e^{2\pi n_{\chi} y}  \( \frac{\pi}{12 c^2 y} \)^{\frac{1}{4}}
\int_{-\infty}^\infty    \frac{1}{(u^2+1)^{\sigma+\frac{1}{8}} } \, dudy\\
&\leq
2.136  \; \frac{\Gamma\(\sigma - \frac{3}{8}\)}{\Gamma\(\sigma + \frac{1}{8}\)}  \sum_{c>0}\frac{|S(1,n,c,\chi)|}{c^{2\sigma+\frac{1}{2}}}\int_{0}^\infty  
 y^{\frac{3}{4}} e^{2\pi n_{\chi} y}
 \, dy,
 \end{align*}
where in the last line we use \cite[3.251 \#2]{table}.
Note that
$$\int_{0}^\infty  
 y^{\frac{3}{4}} e^{2\pi n_{\chi} y} \, dy
 =(-2\pi n_{\chi})^{-\frac{7}{4}} 
\; \Gamma\(\tfrac{7}{4}\).$$
Using \eqref{prop:ZetaSq} and the fact that $\frac{\Gamma\(\delta/2 + 1/8\)}
{\Gamma\(\delta/2 + 5/8\)}$ is decreasing in this range of $\delta$, we get
\begin{align*}
|Q_{1,n}(s,s+2)|&\leq 
2.136 \sum_{c>0}\frac{\tau(c)}{c^{1+\delta}} 
(-2\pi n_{\chi})^{-\frac{7}{4}} 
 \; \frac{\Gamma\(\frac{7}{4}\) \Gamma\(\frac{1}{8}\)}
{\Gamma\(\frac{5}{8}\)}\\
&\leq 0.414  \zeta^{2}(1+\delta)
|n_{\chi}|^{-\frac{7}{4}}. \qedhere
\end{align*}
 \end{proof}

\section{Bounds on $U_m$}
\label{sec:boundsForUm} 
In this section we give a bound for the norm of the Poincar\'e series. The proofs 
use similar techniques as in \cite[Lemma 3.2]{yoshida}. 
We will need the following bounds for the  $K$-Bessel function.
\begin{lemma}
\label{lem:KBessel}
For $y>0$ we have
\be
K_0(y) < 0.975 y^{-\frac{1}{2}}
 \ee
and 
\be
K_0(y) < 1.7 y^{-\frac{7}{2}}.
\ee
\end{lemma}
\begin{proof}
From \cite[(6.28)]{luke} we have
\bee
K_{0}(y)<
 \frac{\sqrt{\frac{\pi}{2}}(16y+7)}{e^{y} y^{\frac{1}{2}}(16y+9)}
 \eee 
for $y>0$.
Since $\frac{\sqrt{\pi}(16y+7)}{\sqrt{2}e^{y}(16y+9)}< \frac{7 \sqrt{\pi}}{9\sqrt{2}}$, the
 first inequality follows.
To obtain the second inequality, we use that 
$e^y > 0.74 y^{3}$. Then
\bee K_{0}(y)<
 \frac{\sqrt{\pi}(16y+7)}{0.74\sqrt{2} y^{\frac{7}{2}}(16y+9)}
 <\frac{\sqrt{\pi}}{0.74\sqrt{2} y^{\frac{7}{2}}}
  <1.7y^{-\frac{7}{2}}.\qedhere
 \eee
\end{proof}
We will also need the following integral representation of $K_0$ (\cite[(10.32.10)]{dlmf}):
\begin{equation}
\label{eq:intK0}
K_0(z) = \frac{1}{2} \int_{0}^{\infty} \exp\(-t-\frac{z^2}{4t}\) \frac{dt}{t}.
\end{equation}

\begin{proposition}
\label{prop:P1}
For $s=\sigma+it$ with $\sigma = 1/2+\delta/2$ and $0<\delta\leq 1/4$, we have
\bee
\lVert U_1(z,s+1,1/2,\chi) \rVert\leq  4.73.
\eee
\end{proposition}
\begin{proof}
Unfolding as in the proof of \cite[Lemma 1]{pribitkin} we have
\begin{multline*}
\lVert U_1(z,s+1,1/2,\chi) \rVert^2  = \(\frac{\pi}{6}\)^{-1-2\sigma} \Gamma(2\sigma+1) 
+ \sum_{c>0} \frac{S(1,1,c,\chi)}{c^{2s+2}} \\
\times \int_{0}^{\infty} \int_{-\infty}^{\infty} \frac{y^{-2it-1}}{(x^2+1)^{s+1}}
\left[ \frac{x+i}{(x^2+1)^{\frac{1}{2}}}\right]^{-\frac{1}{2}}
 e\(\frac{-1}{24yc^2 (x+i)} - \frac{y(x-i)}{24} \) \, dxdy.
\end{multline*}
Taking absolute values we see that
\begin{align*}
\label{eq:u1absInner}
\lVert U_1(z,s+1,1/2,\chi) \rVert^2  &\leq
4.85
+ \sum_{c>0} \frac{|S(1,1,c,\chi)|}{c^{2\sigma+2}}\\
&\times \int_{0}^{\infty} \int_{-\infty}^{\infty} \frac{y^{-1}}{(x^2+1)^{\sigma+1}}
 \exp\(\frac{-\pi y}{12}- \frac{\pi }{12yc^2(x^2+1)}\)
\, dxdy\\
&\leq
2\int_{-\infty}^{\infty} 
\frac{1}{(x^2+1)^{\sigma+1}}
K_0\Big(\frac{\pi }{6c(x^2+1)^{\frac{1}{2}}}\Big)
\, dx,
\end{align*}
where in the last inequality we used \eqref{eq:intK0}.
Using  Lemma \ref{lem:KBessel} we obtain
\bee
\lVert U_1(z,s+1,1/2,\chi) \rVert^2  \leq
4.85+ 6.46  \sum_{c>0} \frac{|S(1,1,c,\chi)|}{c^{2\sigma+2}}c^{\frac{1}{2}}.
\eee
The result follows by \eqref{eq:lehmerAc} and 
 \eqref{prop:ZetaSq}.
\end{proof}

 \begin{proposition}
 \label{prop:norm1-n}
 Let $s=\sigma+it$ with $\sigma=1/2+\delta/2$ and $0<\delta\leq 1/4$. For any integer $n\leq 0$ we have
 \bee
\lVert U_{1-n}(z,s+2,-1/2,\overline{\chi})\rVert
\leq 0.156 \zeta(1+\delta) \tau((1-24n)^2) |n_{\chi}|^{-\frac{7}{4}} .
 \eee
 \end{proposition}
\begin{proof}
Recalling that $(1-n)_{\overline{\chi}}=-n_{\chi} = |n_{\chi}|$, we have
\begin{multline*}
\lVert U_{1-n}(z,s+2,-1/2,\overline{\chi})\rVert^2 = 
 (4\pi |n_{\chi}|)^{-3-2\sigma} \Gamma(2\sigma+3)
+ \sum_{c>0} \frac{S(1-n,1-n,c,\overline{\chi})}{c^{2s+4}} \\
\times \int_{0}^{\infty} \int_{-\infty}^{\infty} \frac{y^{-2it-1}}{(x^2+1)^{s+2}}
\left[ \frac{x+i}{(x^2+1)^{\frac{1}{2}}}\right]^{\frac{1}{2}}  
e\(\frac{-|n_{\chi}| }{yc^2 (x+i)} - |n_{\chi}| \cdot y(x-i)\) \, dxdy.
\end{multline*}
Since $|n_{\chi}| \geq 23/24$, we see that the absolute value is bounded by
\begin{multline}
\label{eq:unabsInner}
\frac{3}{128\pi^4} |n_{\chi}|^{-4} 
+ \sum_{c>0} \frac{|S(1-n,1-n,c,\overline{\chi})|}{c^{2\sigma+4}} 
\int_{0}^{\infty} \int_{-\infty}^{\infty} \frac{y^{-1}}{(x^2+1)^{\sigma+2}}\\
\times \exp(2\pi n_{\chi} y) 
\exp\(\frac{2\pi n_{\chi}}{yc^2(x^2+1)}\) \, dxdy.
\end{multline}
By  \eqref{eq:intK0}, the double integral in   \eqref{eq:unabsInner}
becomes
\be
\label{eq:doubleint1-n}
2\int_{-\infty}^{\infty} 
\frac{1}{(x^2+1)^{\sigma+2}}
K_0\(\frac{4\pi |n_{\chi}|}{c(x^2+1)^{\frac{1}{2}}}\)
\, dx.
\ee
Using Lemma \ref{lem:KBessel} and estimating with $\sigma=1/2$, we see that \eqref{eq:doubleint1-n} is bounded by 
$0.0026 \( \tfrac{c}{|n_{\chi}|} \)^{\frac{7}{2}}$.
Thus,
\be
\label{eq:u1-nIneqK}
\lVert U_{1-n}(z,s+2,-1/2,\overline{\chi})\rVert^2
 \leq
\frac{3}{128\pi^4} |n_{\chi}|^{-4} 
+ 0.0026 |n_{\chi}|^{-\frac{7}{2}}  \sum_{c>0} 
\frac{|S(1-n,1-n,c,\chi)|}{c^{\frac{3}{2}+\delta}} .
\ee
From an argument as in \cite[page 413]{iwkow} using
 \cite[(2.27), (2.29)]{AAimrn} we get 
\bee
\label{eq:kloostermanIAA}
\sum_{c=1}^{\infty} \frac{|S(1-n,1-n,c,\chi)|}{c^{\frac{3}{2}+\delta}} 
\leq \frac{16}{\sqrt{3}} \, \zeta^{2}(1+\delta) \tau((1-24n)^2)^2.
\eee
Note that  $|n_{\chi}|\geq 23/24$. From this and  
\eqref{eq:u1-nIneqK}
we obtain
\begin{align*}
\lVert U_{1-n}(z,s+2,-1/2,\overline{\chi})\rVert^2 &\leq \frac{3}{128\pi^4} |n_{\chi}|^{-4} 
 + 0.0241 |n_{\chi}|^{-\frac{7}{2}}  \zeta^{2}(1+\delta) \tau((1-24n)^2)^2\\
&\leq 0.0242 |n_{\chi}|^{-\frac{7}{2}} \zeta^{2}(1+\delta) \tau((1-24n)^2)^2.
\end{align*}
Proposition \ref{prop:norm1-n} follows.
\end{proof}

\section{Bounds for the Kloosterman zeta function}
\label{sec:boundsForKZeta}
Now we can give a bound for the Kloosterman zeta function. 
Pribitkin's main theorem  \cite{pribitkin} is an ineffective version of this result valid in more generality. 
In this section we will prove the following theorem.
\begin{theorem}
\label{th:KZbound}
Let $n\leq 0$ and $s = \sigma+it$ with $\sigma = 1/2+\delta/2$ and $0<\delta\leq 1/4$. Then
\bee
|Z_{1,n}(s)|\leq
189.91 \zeta^2(1+\delta)
\tau((1-24n)^2)
(1+|t|)^{\frac{1}{2}}
|n_{\chi}|^{\frac{1}{4}}.
\eee
\end{theorem}
The proof of Theorem \ref{th:KZbound} requires some
preliminary results.
 Let $\mathcal{L}_{\frac{1}{2}}(N,\chi)$ 
denote the $L^2$-space of automorphic functions with respect to the Petersson inner product
 given by 
\eqref{eq:inner_prod}.
Define
$\Delta_{\frac{1}{2}} \defeq y^2 \( \frac{\partial^2}{\partial x^2} +\frac{\partial^2}{\partial y^2}\) 
-\frac{i y}{2} \frac{\partial}{\partial x}$.
Then $\Delta_{\frac{1}{2}}$  has a unique self-adjoint extension to $\mathcal{L}_{\frac{1}{2}}(N,\chi)$.
Denote by $\lambda_{0}(1/2) \leq \lambda_{1}(1/2) \leq \cdots$ the discrete spectrum of
$\Delta_{\frac{1}{2}}$.
From \cite[Prop. 1.2]{sarnakAdditive} we have $\lambda_0(1/2) = 3/16$.

 By Lemma \ref{lem:innerprod} we find that
\bee
Z_{1,n}(s)
=
(I_{1,n}(s,w)-Q_{1,n}(s,w))
\frac{\Gamma\(w+\frac{1}{4}\)\Gamma\(s-\frac{1}{4}\)}{\Gamma(w+s-1)\Gamma(w-s)}
(-1)^{s+w}
\pi^{w-s-1}
(-i)^{\frac{1}{2}}
4^{w-1}
|n_{\chi}|^{w-s}.\\
\eee
Now let $w=s+2$. Then
\be
\label{eq:absZexpr}
|Z_{1,n}(s)|
\leq
(|I_{1,n}(s,s+2)|+|Q_{1,n}(s,s+2)|)
\frac{|\Gamma\(s+\frac{9}{4}\)\Gamma\(s-\frac{1}{4}\)|}{|\Gamma(2s+1)|}
4^{\sigma+1}
\pi
|n_{\chi}|^{2}.
\ee

We have the following bound for $|I_{1,n}(s,s+2)|$.
\begin{proposition}
\label{prop:Ibound}
Let  $s=\sigma+it = 1/2+\delta/2+it$ with $0<\delta\leq 1/4$.
Then
\bee
|I_{1,n}(s,s+2)| \leq 0.674 \tau((1-24n)^2) \zeta^2(1+\delta)|n_{\chi}|^{-\frac{7}{4}}.
\eee
\end{proposition}
\begin{proof}
From \cite[(2.4)]{gs} and \cite[(A.2.9)]{sarnakapp} we see that
\begin{align*}
|I_{1,n}(s,s+2)|
&=|\langle   \mathcal{U}_1(z,s,1/2,\chi), \overline{\mathcal{U}_{1-n}(z,s+2,-1/2,\overline{\chi})} \rangle| \\
&\leq (\pi/6) |s-1/4|  | R_{s(1-s)} |  \lVert \mathcal{U}_1(z,s+1,1/2,\chi) \rVert  \;  
\lVert \mathcal{U}_{1-n}(z,s+2,-1/2,\overline{\chi}) \rVert\\
&\leq  \frac{ (\pi/6) |s-1/4| }{\text{distance}(s(1-s),\text{spectrum}(\Delta_{\frac{1}{2}}))}  \lVert \mathcal{U}_1(z,s+1,1/2,\chi) \rVert  \;  
\lVert \mathcal{U}_{1-n}(z,s+2,-1/2,\overline{\chi}) \rVert,
\end{align*}
where
$R_{s(1-s)} = (\Delta_{\frac{1}{2}} + s(1-s))^{-1}$
is the resolvent of 
$\Delta_{\frac{1}{2}}$.

For $|t|>1$ we see that $|s-1/4| \leq 1.07 |t|$ and 
as in \cite{gs},
\bee
\text{dist}(s(1-s),\text{spectrum}(\Delta_{\frac{1}{2}}))\geq|t(2\sigma-1)|.
\eee
Also we have 
$\zeta(1+\delta)/\delta \leq \zeta^2(1+\delta)$, so 
Propositions \ref{prop:P1} and \ref{prop:norm1-n} give us
\bee
|I_{1,n}(s,s+2)|
\leq 0.414\tau((1-24n)^2) \zeta^2(1+\delta) |n_{\chi}|^{-\frac{7}{4}}.
\eee

  For $|t|\leq 1$ we see that $\re(s(1-s))\leq 5/4$.
  Since $\lambda_1(1/2) > 3.86$ (\cite[Corollary 5.3]{AAimrn}) and $\lambda_0(1/2) =3/16$
we have
\bee
\text{distance}(s(1-s),\text{spectrum}(\Delta_{\frac{1}{2}})) 
= \la s(1-s)-\tfrac{3}{16} \ra
= \la s-1/4\ra \la s-3/4 \ra  
\geq \tfrac{1}{8} |s-1/4|.
\eee
Since $\zeta(1+\delta) \geq \zeta(\frac{5}{4}) \geq  4.59$, for these values of $t$ we have
\bee
|I_{1,n}(s,s+2)|
\leq 0.674 \tau((1-24n)^2) \zeta^2(1+\delta)|n_{\chi}|^{-\frac{7}{4}}. \qedhere
\eee 
\end{proof}

To obtain a bound for the $\Gamma$ functions appearing in \eqref{eq:absZexpr},
we use the following lemma.
\begin{lemma}[\cite{rademacherbook}, \S 34, Thm.  A]
\label{lem:radGammaQ}
Let $0 \leq c \leq 1$.Then for $\re(s) \geq (1-c)/2$
we have
\bee
\la \frac{\Gamma(s+c)}{\Gamma(s)} \ra \leq |s|^{c}.
\eee 
\end{lemma}

\begin{proposition}
\label{prop:gammabounds}
Let  $s=\sigma+it = 1/2+\delta/2+it$ with $0<\delta\leq 1/4$.
Then
$$\frac{\la \Gamma\(s+\frac{9}{4}\)\Gamma\(s-\frac{1}{4}\)\ra}{|\Gamma(2s+1)|}
\leq 
5.84(1+|t|)^{\frac{1}{2}}.$$
\end{proposition}
\begin{proof}
Apply the duplication formula
to obtain
$$ |\Gamma(2s+1)|
=
|\Gamma(s+1/2) \Gamma(s+1)| \pi^{-\frac{1}{2}} 4^{\sigma}.$$
Use the functional equation to obtain
\bee
\Gamma\(s-\tfrac{1}{4}\) = \(s-\tfrac{1}{4}\)^{-1} \Gamma\(s+\tfrac{3}{4}\),
\eee
\bee
\Gamma\(s+\tfrac{9}{4}\) = \(s+\tfrac{5}{4}\) \Gamma\(s+\tfrac{5}{4}\).
\eee
Then by Lemma \ref{lem:radGammaQ},
\begin{align*}
\frac{\la \Gamma(s+\frac{9}{4})\Gamma\(s-\frac{1}{4}\)\ra}{|\Gamma(2s+1)|}
&=
\pi^{\frac{1}{2}} 4^{-\sigma}
\frac{\la s+\frac{5}{4} \ra \la \Gamma\(s+\frac{5}{4}\)  \Gamma\(s+\frac{3}{4}\)\ra}
{\la s-\frac{1}{4} \ra  \la\Gamma(s+1)| \Gamma\(s+\frac{1}{2}\)\ra}\\
 &\leq \pi^{\frac{1}{2}} 4^{-\frac{1}{2}}
 \frac{\la s+\frac{5}{4}\ra  |s+1|^{\frac{1}{4}} \la s+\frac{1}{2}\ra^{\frac{1}{4}} }{\la s-\frac{1}{4} \ra}.
 \end{align*}
 If $|t|\leq 1$ we see that 
$\frac{\la \Gamma\(s+\frac{9}{4}\)\Gamma\(s-\frac{1}{4}\)\ra}{|\Gamma(2s+1)|}\leq 
\frac{\la \Gamma\(\frac{11}{4}\)\Gamma\(\frac{1}{4}\)\ra}{|\Gamma(2)|}\leq5.84$.
If $|t|\geq 1$ we have  
\bee
\frac{\la \Gamma\(s+\frac{9}{4}\)\Gamma\(s-\frac{1}{4}\) \ra}{|\Gamma(2s+1)|}
\leq \pi^{\frac{1}{2}} 4^{-\frac{1}{2}}
 \frac{\la s+\frac{5}{4}\ra^{\frac{3}{2}}}{\la s-\frac{1}{4}\ra} \\
 \leq\pi^{\frac{1}{2}} 4^{-\frac{1}{2}}
\la 1 + \frac{6}{4s-1} \ra^{\frac{3}{2}} \la s-\tfrac{1}{4}\ra^{\frac{1}{2}}\\
 \leq 2.51|t|^{\frac{1}{2}}.
\eee The proposition follows. 
 \end{proof}
\noindent {\it Proof of Theorem \ref{th:KZbound}.}
The theorem follows from Lemma \ref{lem:Qbound}, \eqref{eq:absZexpr} and Propositions 
\ref{prop:Ibound} and \ref{prop:gammabounds}. \hfill $\square$

\section{Proof of Theorem \ref{thm:KloostermanBound}}
\label{sec:proofOfMainTh}
We
use Perron's formula as in
\cite[\S 17]{davenport}. Let $f(s) = Z_{1,n}\(\frac{1+s}{2}\)$.
We see that
\be
\sum_{c\leq x} \frac{S(1,n,c,\chi)}{c}
=
\frac{1}{2\pi i} \int_{v-i\infty}^{v+i\infty} f(s) \frac{x^s}{s} ds,
\ee
where $v>1/2$.
Now for $T>0$ and $x\in\mathbb{Z}+1/2$ we have
\begin{multline}
\label{eq:SumSInt}
 \la \sum_{c\leq x} \frac{S(1,n,c,\chi)}{c}
- \frac{1}{2\pi i} \int_{v-iT}^{v+iT} f(s)  \frac{x^s}{s} ds \ra 
=
\frac{1}{2\pi} \la\int_{v-i\infty}^{v+i\infty} f(s) \frac{x^s}{s} ds
- \int_{v-iT}^{v+iT} f(s) \frac{x^s}{s} ds \ra \\
< \sum_{c=1}^{\infty} \frac{|S(1,n,c,\chi)|}{c} \(\frac{x}{c}\)^v 
\min\(1,T^{-1} \la \log\frac{x}{c}\ra^{-1}\).
\end{multline}
Let $v = 1/2+\delta$ where $\delta$ is as in Theorem \ref{thm:KloostermanBound}. Then
\bee
\sum_{c=1}^{\infty} \frac{|S(1,n,c,\chi)|}{c} \(\frac{x}{c}\)^v 
\min(1,T^{-1} \la \log\frac{x}{c}\ra^{-1})\\
 \leq \frac{x^{\frac{1}{2}+\delta}}{T}
 \sum_{c=1}^{\infty} \frac{|S(1,n,c,\chi)|}{c^{\frac{3}{2}+\delta}}
\la \log\frac{x}{c}\ra^{-1}.
\eee
Now we split the sum into the ranges $c\leq \frac{3}{4} x$, $\frac{3}{4} x < c<x$, 
$x<c<\frac{5}{4} x$, and $c\geq \frac{5}{4} x$.
Note that
$\frac{1}{\log x-\log c} \leq \frac{x}{x-c}$ when $c<x$.
So for $x\geq 10000$ we have
\begin{align}
\label{eq:34x}
\sum_{c=\left\lfloor \frac{3x}{4}\right\rfloor+1}^{x-\frac{1}{2}} \frac{|S(1,n,c,\chi)|}{c^{\frac{3}{2}+\delta}}
\la \log\frac{x}{c}\ra^{-1} 
& \leq 
\sum_{c=\left\lfloor \frac{3x}{4} \right\rfloor+1}^{x-\frac{1}{2}} \frac{|S(1,n,c,\chi)|}{c^{\frac{3}{2}+\delta}}
\frac{x}{x-c}\\
&\leq \frac{4\ell(\delta)}{3}  \int_{ \frac{3x}{4} }^{x-\frac{1}{2}}
\frac{dt}{x-t}
+\frac{8\ell(\delta)}{3}\nonumber \\ 
&\leq 1.523\ell(\delta)\log x.\nonumber
\end{align}
Similarly, for $x\geq 10000$ we see that
\be
\label{eq:x54}
\sum_{c=x+\frac{1}{2}}^{\left\lceil \frac{5x}{4}\right\rceil -1} \frac{|S(1,n,c,\chi)|}{c^{\frac{3}{2}+\delta}}
\la \log\frac{x}{c}\ra^{-1} 
\leq \ell(\delta) \sum_{c=x+\frac{1}{2}}^{\left\lceil \frac{5x}{4}\right\rceil -1} 
\frac{1}{c-x}
\leq 1.142 \ell(\delta) \log x.
\ee
 We have
\bee
\sum_{c\leq \frac{3x}{4}} \frac{|S(1,n,c,\chi)|}{c^{\frac{3}{2}+\delta}}
\la \log\frac{x}{c}\ra^{-1} 
\leq 
\sum_{c\leq \frac{3x}{4}} \frac{|S(1,n,c,\chi)|}{c^{\frac{3}{2}+\delta}}
\la \log\frac{4}{3} \ra^{-1} \\
 \leq 
3.5\sum_{c\leq \frac{3x}{4}} \frac{\tau(c)}{c^{1+\delta}}
\eee
and 
\bee
\sum_{c\geq \frac{5x}{4}} \frac{|S(1,n,c,\chi)|}{c^{\frac{3}{2}+\delta}}
\la \log\frac{x}{c}\ra^{-1} 
 \leq 
\sum_{c\geq \frac{5x}{4}} \frac{|S(1,n,c,\chi)|}{c^{\frac{3}{2}+\delta}}
\la \log\frac{4}{5}\ra^{-1} 
 \leq 
4.5\sum_{c\geq \frac{5x}{4}} \frac{\tau(c)}{c^{1+\delta}}.
\eee
By  \eqref{prop:ZetaSq} we obtain
\be
\label{eq:34x54x}
\sum_{c\leq \frac{3x}{4}} +
\sum_{c\geq \frac{5x}{4}} \frac{|S(1,n,c,\chi)|}{c^{\frac{3}{2}+\delta}}
\la \log\frac{x}{c}\ra^{-1} 
 \leq 
4.5\sum_{c=1}^{\infty} \frac{\tau(c)}{c^{1+\delta}}=4.5\zeta^{2}(1+\delta).
\ee
 Therefore, by \eqref{eq:SumSInt},
\eqref{eq:34x}, \eqref{eq:x54}, and \eqref{eq:34x54x} we obtain
 \be
\label{eq:KSumIntBound}
\la \sum_{c\leq x} \frac{S(1,n,c,\chi)}{c} \ra
\leq  \la \int_{\frac{1}{2}+\delta-iT}^{\frac{1}{2}+\delta+iT}f(s)\frac{x^s}{s} ds \ra
+ \(4.5\zeta^2(1+\delta) + 3 \ell(\delta) \log x \) \frac{x^{\frac{1}{2}+\delta}}{T}.
\ee

From \cite[(3.2)]{gs} we see that $Z_{m,n}\(\frac{1+s}{2}\)$  is holomorphic for $\re(s)>0$ (the only possible pole at $s=1/2$ does not arise since $n\leq 0$).
Thus
\be
\label{eq:rectangleIntegral}
\int_{\partial E} Z_{m,n}\(\frac{1+s}{2}\) \frac{x^s}{s} ds = 0,
\ee
where $E$ is the rectangle $[\delta,1/2+\delta]\times[-T,T]$.
We obtain a bound for  $Z_{1,n}\(\frac{1+\delta+it}{2}\)$ using Theorem \ref{th:KZbound} and a bound 
for  $Z_{1,n}\(\frac{\frac{3}{2}+\delta+it}{2}\)$ using the Weil bound.
We require the Phragm\'en-Lindel\"of principle for a strip.
 \begin{proposition}[\cite{iwkow}, Thm. 5.53]
 Let $f$ be a function holomorphic on an open neighborhood of a strip $a\leq \sigma \leq b$, 
 for some real numbers $a < b$, such that $|f(s)| \ll \exp(|s|^{A})$ for 
 some $A\geq 0$ and $a\leq \sigma \leq b$.
Assume that
\begin{align*}
& |f(a+it)| \leq M_{a}(1+|t|)^{\alpha}, \\
& |f(b+it)| \leq M_{b}(1+|t|)^{\beta} 
\end{align*}
for $t \in \R$. Then
\bee
|f(\sigma+it)| \leq M_{a}^{d(\sigma)} M_{b}^{1-d(\sigma)} 
(1+|t|)^{\alpha d(\sigma)+\beta(1-d(\sigma))}
\eee
for all $s$ in the strip, where $d$ is the linear function such that $d(a) =1$, $d(b) = 0$.
\end{proposition}
\begin{proposition}
\label{prop:phagmenfbound}
Let $n\leq 0$ and $f(s) = Z_{1,n}(\frac{1+s}{2})$.
For $\delta \leq \sigma \leq 1/2+\delta$ with $0< \delta \leq 1/4$ we have
\bee
|f(\sigma+it)| \leq 189.91  \zeta^{2}(1+\delta)
 \tau((1-24n)^2)  |n_{\chi}|^{\frac{1}{4}} (1+|t/2|)^{-\sigma+\frac{1}{2}+\delta}. 
 \eee
\end{proposition}
\begin{proof}
By Theorem \ref{th:KZbound}, we have 
\be
\label{eq:leftside}
|f(\delta+it)| \leq 
189.91 \zeta^{2}(1+\delta) \tau((1-24n)^2) |n_{\chi}|^{\frac{1}{4}} (1+|t/2|)^{\frac{1}{2}}.
\ee
Also, by \eqref{prop:ZetaSq}
\bee
\la f\(\frac{1}{2}+\delta+it\)\ra=\la\sum_{c=1}^{\infty} \frac{S(1,n,c,\chi)}{c^{\frac{3}{2}+\delta+it}} \ra 
\leq \sum_{c=1}^{\infty} \frac{c^{\frac{1}{2}}\tau(c)}{c^{\frac{3}{2}+\delta} } =\zeta^{2}(1+\delta).
\eee
Note that the line $d$ such that 
$d(\delta) = 1$ 
and 
$d(1/2+\delta) = 0$
is 
$d(\sigma) = -2\sigma+1+2\delta$.
By Phragm\'en-Lindel\"of, for $\delta \leq \sigma \leq 1/2+\delta$ 
we have
\begin{align*}
|f(\sigma+it)| &\leq \(189.91\zeta^{2}\(1+\delta\)  
\tau\(\(1-24n\)^2\)  |n_{\chi}|^{\frac{1}{4}}\)^{-2\sigma+1+2\delta}
\(\zeta^{2}\(1+\delta\)\)^{2\sigma-2\delta}\(1+|t/2|\)^{-\sigma+\frac{1}{2}+\delta}\\
&=\zeta^{2}\(1+\delta\)
 \(189.91 \tau\(\(1-24n\)^2\)  |n_{\chi}|^{\frac{1}{4}}\)^{-2\sigma+1+2\delta}\(1+|t/2|\)^{-\sigma+\frac{1}{2}+\delta}. \qedhere
\end{align*}
\end{proof}

We are now ready to prove Theorem \ref{thm:KloostermanBound}.
  \begin{proof}[Proof of Theorem \ref{thm:KloostermanBound}]
 For  $x< 10000$ this follows from  $\la \sum_{c\leq x} \frac{A_c(n)}{c} \ra \leq x$. 
  Let $x\geq 10000$ and set $T= x^{\frac{1}{3}}$. By Proposition \ref{prop:phagmenfbound} we have
\begin{align*}
&\la \int_{\frac{1}{2}+\delta+iT}^{\delta+iT} f(s) \frac{x^s}{s} ds \ra
= \la \int_{\delta}^{\frac{1}{2}+\delta} f(\sigma+iT) \frac{x^{\sigma+iT}}{\sigma+iT} d\sigma \ra\\
&\leq  189.91  \zeta^{2}(1+\delta) \tau((1-24n)^2)   |n_{\chi}|^{\frac{1}{4}}  (1+T/2)^{\frac{1}{2}+\delta}
\int_{\delta}^{\frac{1}{2}+\delta}
  (1+T/2)^{-\sigma} \frac{x^{\sigma}}{|\sigma+iT|} d\sigma\\
           &\leq \frac{189.91   \zeta^{2}(1+\delta)  |n_{\chi}|^{\frac{1}{4}}\tau((1-24n)^2) (1+T/2)^{\frac{1}{2}+\delta}}{T}    \cdot
     \frac{\(\frac{x}{1+T/2}\)^{\frac{1}{2}+\delta}-\(\frac{x}{ 1+T/2}\)^{\delta}}{\log\(\frac{x}{ 1+T/2}\)}.
        \end{align*}
     Disregarding the negative term and using the estimate $\log\(\frac{x}{1+T/2}\) \geq 6.74$ gives
    \be
    \label{eq:TopBottomIntegral}
    \la \int_{\frac{1}{2}+\delta+iT}^{\delta+iT} f(s) \frac{x^s}{s} ds \ra \leq
    28.18 \zeta^{2}(1+\delta) \tau((1-24n)^2) |n_{\chi}|^{\frac{1}{4}}  x^{\frac{1}{6}+\delta} .
    \ee
    The same bound holds for the bottom of the rectangle.
To estimate the integral  over the left side of the rectangle we use \eqref{eq:leftside}
 to see that
$\int_{\delta+iT}^{\delta-iT} f(s) \frac{x^s}{s} ds$
is bounded by
\be
\label{eq:rectIntLeft}
  189.91 \zeta^2(1+\delta)  \tau((1-24n)^2) |n_{\chi}|^{\frac{1}{4}} x^{\delta}
  \( 2\int_{2}^{T}    \frac{ (1+t/2)^{\frac{1}{2}}}{|\delta+it|} dt
  +
  2\int_{0}^{2}    \frac{ (1+t/2)^{\frac{1}{2}}}{|\delta+it|} dt\).
\ee
Using the inequality 
    $\int_{2}^{T}   \(\frac{1}{2t} + \frac{1}{t^2}\)^{\frac{1}{2}} dt
     \leq  \int_{2}^{T}   \frac{1}{(2t)^{\frac{1}{2}}} dt$
    we see that the first term in the expansion of \eqref{eq:rectIntLeft} is bounded by
\be
537.15 \zeta^{2}(1+\delta) \tau((1-24n)^2) |n_{\chi}|^{\frac{1}{4}} x^{ \delta} T^{\frac{1}{2}}.
\ee
 Note that 
 \bee
 \int_{0}^{2}  \(\frac{1+|t|/2}{\delta^2+t^2}\)^{\frac{1}{2}} \, dt
    \leq \sqrt{2} \int_{0}^{2}  \(\frac{1}{\delta^2+t^2}\)^{\frac{1}{2}} \, dt
    =\frac{1}{\sqrt{2}} \log \(\delta ^2+4 \sqrt{\delta ^2+4}+8\) - \sqrt{2} \log \(\delta \).
\eee
Since $0< \delta \leq 1/4$ we have 
$\frac{1}{\sqrt{2}} \log \(\delta ^2+4 \sqrt{\delta ^2+4}+8\)\leq 1.77$, so
$$ \int_{0}^{2}  \(\frac{1+|t|/2}{\delta^2+t^2}\)^{\frac{1}{2}} \, dt
    \leq  2.7 |\log \delta|.$$
Thus, the second term in the expansion of \eqref{eq:rectIntLeft} is bounded by 
\be
1025.52 \zeta^{2}(1+\delta) \tau((1-24n)^2) |n_{\chi}|^{\frac{1}{4}} x^{\delta}|\log\delta|.
\ee
Thus (recall that $x\geq 10000$ and $0<\delta\leq 1/4$),
\begin{align}
\label{eq:LeftIntegral}
\nonumber
\la \int_{\delta+iT}^{\delta-iT} f(s) \frac{x^s}{s} ds \ra
   &\leq  \zeta^{2}(1+\delta) \tau((1-24n)^2) |n_{\chi}|^{\frac{1}{4}}\(
   537.15 x^{\frac{1}{6}+\delta} 
    + 1025.52 x^{\delta}|\log \delta|\)\\
      &\leq 608.42 \zeta^{2}(1+\delta) \tau((1-24n)^2)  |\log \delta| |n_{\chi}|^{\frac{1}{4}} x^{\frac{1}{6}+\delta}.
   \end{align}
Hence, using  \eqref{eq:rectangleIntegral},\eqref{eq:TopBottomIntegral}, and  \eqref{eq:LeftIntegral}
we get
   \be
   \label{eq:intsum}
 \la \int_{\frac{1}{2}+\delta-iT}^{\frac{1}{2}+\delta+iT} f(s) \frac{x^s}{s} ds \ra
\leq 649.08 \zeta^{2}(1+\delta) \tau((1-24n)^2)  |\log\delta|  |n_{\chi}|^{\frac{1}{4}} x^{\frac{1}{6}+\delta}.
\ee
 The result  follows from \eqref{eq:A-c-n-S} and \eqref{eq:KSumIntBound}.
   \end{proof}

\section{Proof of Theorem \ref{thm:sptbound}}
\label{sec:proofOfSptbound}
We begin with a lemma.
\begin{lemma}
\label{lem:intTaylorexp}
Let $y\geq 2$ and $q\leq 1$. Then 
\bee
\int_{y}^{\lambda(n)}\, e^{\frac{\lambda(n)}{t}} t^{-\frac{3}{2}+q}  \, dt
\leq\frac{4}{3-2q} \lambda(n)^{-1} y^{1/2+q}  
 e^{\frac{\lambda(n)}{y}}.
\eee
\end{lemma}

\begin{proof}
Using the Taylor expansion for $e^{\frac{\lambda(n)}{t}}$ we have
\begin{align*}
\int_{y}^{\lambda(n)}\, e^{\frac{\lambda(n)}{t}} t^{-\frac{3}{2}+q}  \, dt
       &\leq  \lambda(n)^{-1} y^{1/2+q}    \sum_{m=1}^{\infty}
  \frac{\(\frac{\lambda(n)}{y}\)^{m+1}}{(m+1)!} \frac{4}{3-2q}  \\
      &\leq    \frac{4}{3-2q} \lambda(n)^{-1} y^{1/2+q}  
 e^{\frac{\lambda(n)}{y}}. \qedhere
  \end{align*}
\end{proof}

Next is the proof of Theorem \ref{thm:S}. 
\begin{proof}[Proof of Theorem \ref{thm:S}]
By partial summation and \eqref{eq:AASn}, we have
\be
\label{eq:SnPartialSummation}
S(n) = \sqrt{24\pi} \lambda(n)^{\frac{1}{2}} \(I_{\frac{1}{2}}(\lambda(n)) 
- I_{\frac{1}{2}}\(\frac{\lambda(n)}{2}\) - \int_{2}^{\infty} \(I_{\frac{1}{2}}\(\frac{\lambda(n)}{t}\)\)' 
\sum_{c\leq t} \frac{A_c(n)}{c} dt\).
\ee
Note that the convergence of the integral in \eqref{eq:SnPartialSummation}
follows from Theorem \ref{thm:KloostermanBound}.
From  (5.4) of \cite{AA} we have
$$\la \(I_{\frac{1}{2}}\(\frac{\lambda(n)}{x}\)\)' \ra = \frac{\lambda(n)}{2x^2} \(I_{-\frac{1}{2}}\(\frac{\lambda(n)}{x}\) +
I_{\frac{3}{2}}\(\frac{\lambda(n)}{x}\) \),$$
so
\begin{multline}
\label{eq:forS}
\left| \int_{2}^{\infty} \(I_{\frac{1}{2}}\(\frac{\lambda(n)}{t}\)\)' \sum_{c\leq t} \frac{A_c(n)}{c} dt \right|\\
\leq \int_{2}^{\infty}  \frac{\lambda(n)}{2t^2} 
\(I_{-\frac{1}{2}}\(\frac{\lambda(n)}{t}\) +I_{\frac{3}{2}}\(\frac{\lambda(n)}{t}\)\)
\left| \sum_{c\leq t}   \frac{A_c(n)}{c}\right|dt.
\end{multline}
Let $f(x) =  \(I_{-\frac{1}{2}}(x) + I_{\frac{3}{2}}(x)\) x^{\frac{1}{2}}e^x$.
We have $f'(x)>0$, so
$f(x) \leq f(1) < 4.146$ for $x\leq 1$
and
\be
\label{eq:ineg1/2<1}
  I_{-\frac{1}{2}}(x) + I_{\frac{3}{2}}(x) <4.146 x^{-\frac{1}{2}}e^{-x} \qquad \text{ for } x\leq 1.
\ee
Let $g(x) =  \(I_{-\frac{1}{2}}(x) + I_{\frac{3}{2}}(x)\) x^{\frac{1}{2}}e^{-x}$. We have $g'(x) > 0$, so
$g(x) \leq \lim_{x\to\infty} g(x) < 0.798$ for $x\geq 1$. Thus,
\be
\label{eq:ineg1/2>1}
 I_{-\frac{1}{2}}(x) + I_{\frac{3}{2}}(x)<0.798  x^{-\frac{1}{2}} e^{x} \qquad \text{ for } x\geq 1.
\ee
Using  \eqref{eq:forS}, \eqref{eq:ineg1/2<1}, and \eqref{eq:ineg1/2>1}   we see that
\begin{multline} 
\label{eq:int2inftybound}
\la \int_{2}^{\infty} \(I_{\frac{1}{2}}\(\frac{\lambda(n)}{t}\)\)' \sum_{c\leq t} \frac{A_c(n)}{c} dt\ra\\
\leq  0.399\sqrt{\lambda(n)}  \int_{2}^{\lambda(n)}  
  e^{\frac{\lambda(n)}{t}}t^{-\frac{3}{2}} \la  \sum_{c\leq t} \frac{A_c(n)}{c}\ra  dt 
+   2.073 \sqrt{\lambda(n)}  \int_{\lambda(n)}^{\infty} 
  t^{-\frac{3}{2}} \la   \sum_{c\leq t} \frac{A_c(n)}{c}\ra   dt.
  \end{multline}
   Using the trivial bound $| A _c(n)|\leq c$
and Lemma \ref{lem:intTaylorexp}
 we see that 
the first term in \eqref{eq:int2inftybound}
is bounded by
\be
\label{eq:FirstIntegral}
4.515 \, \frac{e^{\frac{\lambda(n)}{2}}}{\sqrt{\lambda(n)}}.
\ee
From Corollary \ref{ex:KloostermanBound1/4}
 we see that the second term in 
\eqref{eq:int2inftybound} is bounded by 
\be
\label{eq:midrdainfty}
  24.88 \(19094.8  \tau((24n-23)^2)  (n-1/24)^{\frac{1}{4}} 
 +25.35(12 +\log \lambda(n))\)  \lambda(n)^{\frac{5}{12}}.
\ee
We can suppose $n\geq 10 000$. Then by \eqref{eq:midrdainfty}
 the second term in 
\eqref{eq:int2inftybound} is bounded by 
\be
\label{eq:rdainfty}
\frac{0.001e^{\frac{\lambda(n)}{2}} }{\sqrt{\lambda(n)}}.
\ee
From  \eqref{eq:int2inftybound},
 \eqref{eq:FirstIntegral}, and \eqref{eq:rdainfty} we obtain
 \bee
\la \int_{2}^{\infty} I_{\frac{1}{2}} \( \frac{\lambda(n)}{t} \)'  \sum_{c\leq t}   \frac{A_c(n)}{c}dt \ra
\leq \frac{4.516e^{\frac{\lambda(n)}{2}}}{\sqrt{\lambda(n)}}.
\eee
Noting that 
\bee
|E_S(n)| \leq \sqrt{24\pi} \lambda(n)^{\frac{1}{2}} \(\la I_{\frac{1}{2}}\(\frac{\lambda(n)}{2}\) \ra
+ \la \int_{2}^{\infty} I_{\frac{1}{2}}\( \frac{\lambda(n)}{t}\)' \sum_{c\leq t} \frac{A_c(n)}{c} \, dt \ra \),
\eee
the Theorem follows for $\lambda(n)>256$. For $\lambda(n)\leq 256$ it can be verified by direct computation.
\end{proof}

We have the following result for $p(n)$.
\begin{lemma}
\label{lem:p}
For $n\geq 1$ we have
\bee
p(n) = \frac{2\sqrt{3}}{24n-1} \(1-\frac{1}{\lambda(n)}\)e^{\lambda(n)} +E_p(n)
\eee
where
\bee
E_{p}(n) \leq \frac{5 e^{\frac{\lambda(n)}{2}}}{24n-1}.
\eee
\end{lemma}
\begin{proof}
For $\lambda(n)>573$ this follows from \cite[(4.14)]{lehmer}. For $\lambda(n)\leq 573$ it can be verified by direct computation.
\end{proof}

\begin{proof}[Proof of Theorem \ref{thm:sptbound}]
The theorem follows by \eqref{eq:sptSp}, Theorem \ref{thm:S}, and Lemma~\ref{lem:p}. 
\end{proof}

\begin{proof}[Proof of Theorem \ref{thm:sptbound2}]
The proof is similar to the proof of Theorem \ref{thm:sptbound}.
\end{proof}

\begin{proof}[Proof of Corollary \ref{cor:spt}] 
  From Lemma \ref{lem:p} we see that \bee
  \frac{\sqrt{24n-1}}{2\pi}  p(n) = 
  \frac{\sqrt{3}}{\pi\sqrt{24n-1}}e^{\lambda(n)} 
  -  \frac{\sqrt{3}}{\pi\sqrt{24n-1}} \cdot \frac{e^{\lambda(n)}}{\lambda(n)}
  +\frac{\sqrt{24n-1}}{2\pi} E_p(n). 
  \eee
  By Lemma \ref{lem:p} and
Theorem \ref{thm:sptbound} the result follows.
  \end{proof}

\noindent {\bf Acknowledgements.}
The author thanks Scott Ahlgren for many
useful suggestions 
and 
Frank Garvan for a helpful conversation about computing values of $\spt(n)$.
The author was partially supported by
the Alfred P. Sloan Foundation's MPHD Program, awarded in 2017.

\bibliographystyle{alpha}
\bibliography{spt}

\end{document}